\documentclass[11pt]{amsart}

\usepackage{epsfig,amsmath,amsfonts,latexsym}
\usepackage[abs]{overpic}
\usepackage{color}
 \usepackage{palatino}
 \usepackage{hyperref}

\newtheorem{question}{Question}
\newtheorem{theorem}{Theorem}
\newtheorem{proposition}[theorem]{Proposition}
\newtheorem{lemma}[theorem]{Lemma}
\newtheorem{corollary}[theorem]{Corollary}

\theoremstyle{definition}

\theoremstyle{remark}
\newtheorem{remark}[theorem]{Remark}

\def\dfn#1{{\em #1}}
\def\R{\mathbb{R}}
\def\Z{\mathbb{Z}}
\def\C{\mathbb{C}}
\def\N{\mathbb{N}}

\numberwithin{theorem}{section}
\theoremstyle{plain}

\begin{document}

\title[Fillings and contact surgery]{Symplectic fillings, contact surgeries, \\and Lagrangian disks}

\author{James Conway}

\author{John B. Etnyre}

\author{B\"{u}lent Tosun}

\address{Department of Mathematics \\ University of California, Berkeley \\ Berkeley \\ California}

\email{conway@berkeley.edu}

\address{School of Mathematics \\ Georgia Institute
of Technology \\  Atlanta  \\ Georgia}

\email{etnyre@math.gatech.edu}

\address{Department of Mathematics\\ University of Alabama\\Tuscaloosa\\Alabama}

\email{btosun@ua.edu}


\begin{abstract}
This paper completely answers the question of when contact $(r)$--surgery on a Legendrian knot in the standard contact structure on $S^3$ yields a symplectically fillable contact manifold for $r\in(0,1]$. We also give obstructions for other positive $r$ and investigate Lagrangian fillings of Legendrian knots. 
\end{abstract}

\maketitle

\section{Introduction}

An interesting and much studied question asks what properties are preserved under various types of contact surgeries. This has been extensively studied for contact $(-1)$--surgeries, that is, Legendrian surgery. For example it is known that any type of symplectic fillability is preserved \cite{EtnyreHonda02a, Weinstein91}, as is non-vanishing of the Heegaard Floer invariant \cite{OzsvathSzabo05a}, and tightness \cite{Wand15}; on the other hand universal tightness is not preserved \cite{Gompf98}.  Less is known about positive contact surgeries, though there are some results about the non-vanishing of the Heegaard Floer contact invariant \cite{Golla15, MarkTosunPre}. In this paper we completely answer the question of when fillability is preserved under $(r)$--contact surgeries on knots in the standard contact structure on $S^3$, for $r\in(0,1]$.  We also give many examples and constructions of Legendrian knots on which contact $(+1)$--surgery yields a symplectically fillable contact structure. We also discuss obstructions to certain contact surgeries being symplectically fillable. 

There are several results showing that contact $(+1)$--surgery on certain knots results in a symplectically non-fillable contact structure \cite{KalotiTosun??, LiscaStipsicz04, OwensStrle12}. At the 2017 International Georgia Topology conference, Ko Honda asked if there were any Legendrian knots in $S^3$ on which contact $(+1)$--surgery produces a fillable contact structure. The only obvious such knot, and only such knot many experts at the conference were able to come up with, was the maximal Thurston--Bennequin unknot. Further investigation yields the following complete answer.  

\begin{theorem}\label{thm1}
Let $L$ be a Legendrian knot in $(S^3,\xi_{std})$. For $r\in (0,1]$, contact $(r)$--surgery on $L$ is strongly symplectically fillable if and only if $r=1$ and $L$ bounds a Lagrangian disk in $(B^4,\omega_{std})$. 

Any such minimal filling will be an exact symplectic filling and have the homology of $S^1\times D^3$. Moreover, the filling can be taken to be a Stein filling if and only if $L$ bounds a regular Lagrangian disk in $(B^4,\omega_{std})$. In particular, if $L$ bounds a decomposable Lagrangian disk in $(B^4,\omega_{std})$, then the filling can be taken to be Stein. 
\end{theorem}
Recall that a Lagrangian disk is decomposable if it can be constructed from maximal Thurston--Bennequin unknots by ``pinch moves" and Legendrian isotopy and a Lagrangian disk in a Weinstein manifold is regular if there is a Liouville vector field for the symplectic form that is tangent to the Lagrangian disk. See Section~\ref{dls} for more details on decomposable and regular Lagrangian submanifolds. There we will show that every decomposable disk is regular, but it is not known if Lagrangian disks must be regular, or if regular disks must be decomposable. 

\begin{remark}
The proof of Theorem~\ref{thm1} allows the following generalization to weak symplectic fillings
: for $r \in (0, 1]$, contact $(r)$--surgery on $L$ is weakly symplectically fillable if and only if $r = 1$ and $L$ bounds a null-homologous Lagrangian disk in some blow-up of $(B^4, \omega_{std})$.  This implies that $tb(L) = -1$, $rot(L) = 0$, and $\tau(L) = 0$.  We do not know whether there are any examples of such $L$ that do not also bound Lagrangian disks in $(B^4, \omega_{std})$.
\end{remark}

\begin{remark}
If $r$ is not $1/n$ for some integer $n$, then contact $(r)$--surgery is not uniquely defined but depends on choices of stabilizations of the Legendrian knot and its Legendrian push-offs \cite{DingGeigesStipsicz04}. We are particularly interested in the choice corresponding to inadmissible transverse surgery \cite{Conway14?}, as this seems the most natural and is the one most studied  \cite{Golla15, LiscaStipsicz04, MarkTosunPre}.  
Specifically, we choose all negative stabilizations for contact $(r)$--surgery, and we will always be considering this contact structure when discussing contact $(r)$--surgery. 
\end{remark}

\begin{remark}
The $r=1$ case of the theorem will probably not be a surprise to the expert reader. The proof is largely straight-forward, and except for some normalization near the boundary of the Lagrangian disk coming from \cite{EtnyreLidmanNgPre} and cited in  Lemma~\ref{mt}, the arguments largely rely on standard regular neighborhood theorems similar to those in \cite{Weinstein71}. The real interest in this case is in realizing that this is the the right statement to solve the filling problem and that it can be generalized to $r\in(0,1]$ and partially generalized to $r>1$, see below. 
\end{remark}

This theorem immediately gives the following obstructions to a knot having a strongly fillable contact $(+1)$--surgery. 
\begin{corollary}
If contact $(+1)$--surgery on $L$ is symplectically fillable, then 
\begin{enumerate}
\item $tb(L)=-1$ and $rot(L)=0$,
\item the knot type of $L$ is quasi-positive,
\item  the knot type of $L$ is slice, and
\item $\tau(L)=0$ and $\epsilon(L) = 0$,
\end{enumerate}
where $\tau(L)$ and $\epsilon(L)$ are Heegaard Floer concordance invariants of the smooth knot type of $L$.
\end{corollary}
\begin{proof}
Since $L$ bounds a Lagrangian disk in the 4--ball it is clearly slice. It is also easy to compute that the Thurston--Bennequin invariant is $-1$ and the rotation class is $0$, see \cite{Chantraine10}. Since $L$ is slice, $\tau(L)=0$ and $\epsilon(L) = 0$ by \cite{Hom14}. A Lagrangian filling of a Legendrian knot can be perturbed to be a symplectic filling (and hence a complex filling) of the transverse push-off of $L$, so $L$ is quasi-positive, see \cite{CaoGallupHaydenSabloff14}.
\end{proof}

\subsection{Legendrian knots bounding Lagrangian disks}
As indicated above, there are many knots that satisfy the condition of Theorem~\ref{thm1}, although there are not too many with small crossing number.  
In \cite{CornwellNgSivek16}  it was shown that for knots with 12 or fewer crossings, the only examples are $\overline{9_{46}}$, $\overline{10_{140}}, 11n_{139}, 12n_{582}, \overline{12n_{768}},$ and $12n_{838}$, see Figure~\ref{m946}. Here, the names are the ones given in KnotInfo \cite{ChaLivingston}, and a bar over the name indicates the mirror of that knot. There are 17 knots with 13 or 14 crossings that have Legendrian representatives that bound decomposable Lagrangian disks. Moreover, starting from these small crossing knots, one can sometime produce infinitely many examples of knots that satisfy the conditions of Theorem~\ref{thm1}. For example, in the knot diagram of $\overline{9_{46}}$ (which is also known as the Pretzel knot $P(-3,-3,3)$) shown in Figure~\ref{m946}, one can introduce $m\geq 0$ half-ribbon twists in the right upper crossing to obtain an infinite family of knots, $P(-3-m,-3,3)$, each of which bounds a Lagrangian disk.    

\begin{figure}[h]
\tiny
\begin{overpic}
{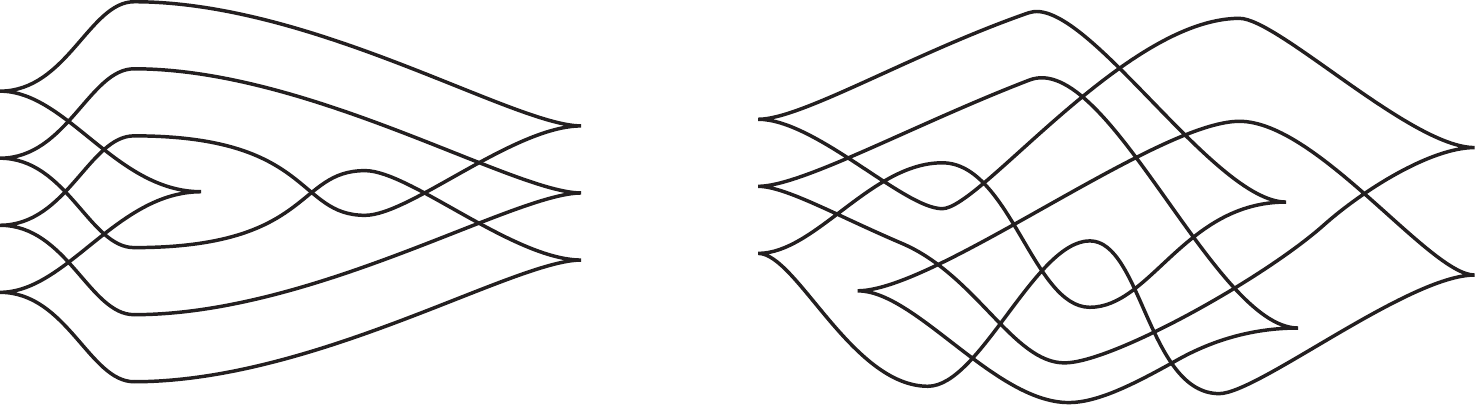}
\end{overpic}
\caption{Mirror of the $9_{46}$ knot and the $12n_{838}$ knot.}
\label{m946}
\end{figure}    

\begin{remark}
There are many examples of quasi-positive slice knots with maximal Thurston--Bennequin invariant less than $-1$, for example the knots $8_{20}$ and $10_{155}$. These knots bound complex disks in the 4--ball, but not Lagrangian disks.  
\end{remark}

In the proof of Theorem~\ref{thm1} we will see that the symplectic filling of contact $(+1)$--surgery on a Legendrian knot $L$ that bounds a Lagrangian disk $D$ in $(B^4, \omega_{std})$ is constructed by removing a neighborhood of $D$ from $B^4$. So it is interesting to wonder whether or not distinct Lagrangian disks with boundary $L$ can lead to distinct symplectic fillings. For example, \cite{EliashbergPolterovich96} shows that the maximal Thurston-Bennequin unknot essentially has one Lagrangian filling disk (at least if a suitably large portion of the symplectization of $S^3$ is added to $B^4$) and it is also known that contact $(+1)$--surgery on the maximal Thurston-Bennequin unknot has a unique Stein filling up to symplectomorphism. In contrast,  Ekholm \cite{Ekholm2016} proved that the knot $\overline{9_{46}}$ has two Lagrangian disk fillings which are not Hamiltonian isotopic. It is known that the complements of neighborhoods of these two ribbon disks are diffeomorphic 4--manifolds, but one naturally wonders whether or not these disks give rise to non-symplectomorphic fillings of contact $(+1)$--surgery on the Legendrian knot $\overline{9_{46}}$. We generalize this in the form of the following question.

\begin{question}
Let $L$ be a Legendrian knot in $(S^3, \xi_{std})$ with two (or more) distinct Lagrangian disk fillings in $(B^4,\omega_{std})$. Does contact $(+1)$--surgery on $L$ have more than one Stein (or symplectic) filling up to symplectomorphism? 
\end{question}

The connected sum and certain cables of Legendrian knots bounding Lagrangian disks also bound Lagrangian disks. Specifically, it is easy to prove the following results. The former seems to have first been noticed by Ekholm, Honda, and K\'alm\'an \cite{EkholmHondaKalman16} when they introduced the notion of decomposable cobordisms, while the latter follows directly from an observation of Cornwell, Ng, and Sivek \cite{CornwellNgSivek16} in the non-decomposable case and an argument of Liu, Sabloff, Yacavone, and Zhou \cite{LiuSabloffYacavoneZhouPre} in the decomposable case. 
\begin{proposition}[Ekholm, Honda, and Kalman 2016, \cite{EkholmHondaKalman16}]\label{connectsum}
If $L$ and $L'$ are Legendrian knots in the boundary of a symplectic manifold with convex boundary $(X,\omega)$, and they bound disjoint (decomposable) Lagrangian disks, then $L\# L'$ also bounds a (decomposable) Lagrangian disk. 
\end{proposition}
For cables we have the following. 
\begin{proposition}\label{cable}
Let $(M,\xi)$ be a contact 3--manifold manifold with symplectic filling $(X,\omega)$. If $L$ is a Legendrian knot in $(M,\xi)$ that bounds a (decomposable) Lagrangian disk in $(X,\omega)$, then the $(n,1)$--cable of $L$ also bounds a (decomposable) Lagrangian disk in $(X,\omega)$. 
\end{proposition}
An example of this is the $(3,1)$--cable of the mirror of the $9_{46}$ knot, shown in Figure~\ref{cablem946}.
\begin{figure}[h]
\tiny
\begin{overpic}
{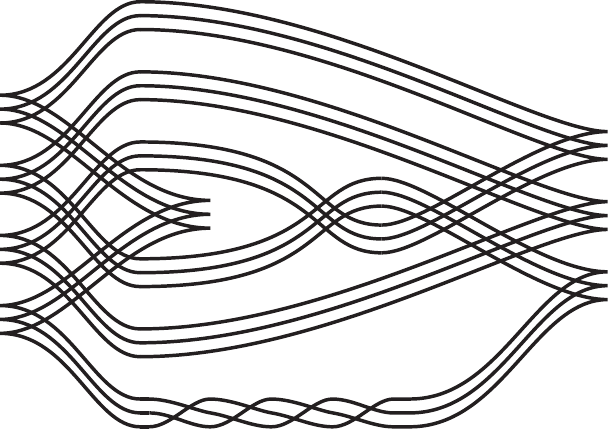}
\end{overpic}
\caption{The $(3,1)$-cable of the mirror of the $9_{46}$ knot.}
\label{cablem946}
\end{figure}  

In \cite{Yasui15preprint}, Yasui showed how to build other Legendrian knots that bound Lagrangian disks. This technique was further explored by McCullough in \cite{McCulloughPre}. The latter work starts by proving the following ``folk result" that all ribbon knots can be described as follows. 
\begin{theorem}\label{ribbon}
Let $K$ be a ribbon knot in $S^3$. There is a handle presentation of $B^4$ consisting of one 0--handle, and $n$ cancelling $1$, $2$--handle pairs. In this handlebody, there is an unknot in the boundary of the 0--handle that is disjoint from the $1$ and $2$--handles, such that when the $1$ and $2$--handles are cancelled, the unknot becomes $K$.
\end{theorem}
Theorem~\ref{ribbon} was inspired by Yasui's example shown in Figure~\ref{handle} that also illustrates what the theorem says.  
\begin{figure}[h]
\begin{overpic}
{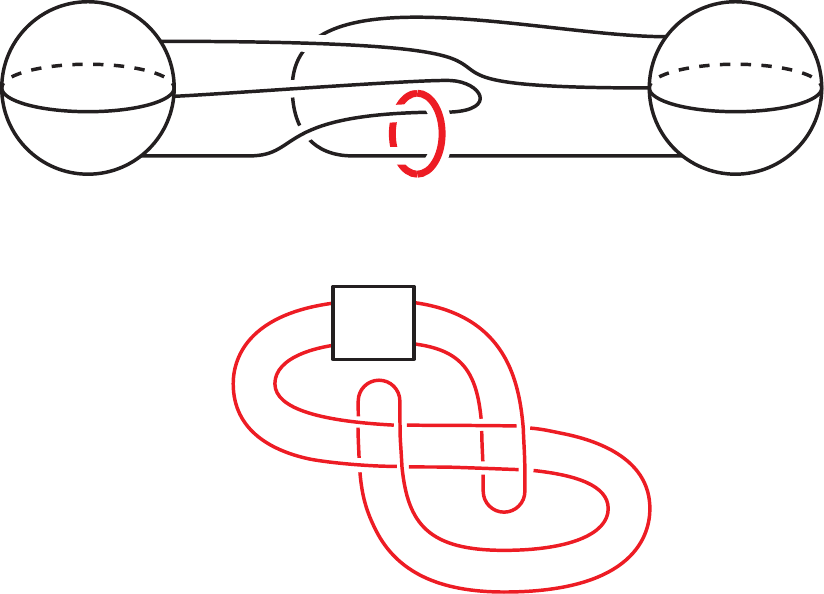}
\put(60,165){$m$}
\put(103,75){$m$}
\end{overpic}
\caption{The top figure shows a cancelling 1, 2--handle pair in a handle decomposition of $B^4$ together with an unknot in the boundary of the 0--handle. The bottom figure shows the what becomes of the unknot after the 1 and 2--handles are canceled.}
\label{handle}
\end{figure}  
Notice that the complement of the slice disk has a handle presentation obtained by turning the unknot in the diagram into a 1--handle by putting a dot on the unknot (that is using dotted circle notation for 1--handles). 

Given a presentation for a ribbon knot as in Theorem~\ref{ribbon}, then if the link to which the 2--handles are attached can be Legendrian realized so that Legendrian surgery gives the desired smooth surgery, then we have a Stein presentation of the standard symplectic structure on the 4-ball. Moreover, if the above can be done so that the unknot in the picture can be Legendrian realized with $tb=-1$, then it will bound a Lagrangian disk in the Stein manifold (since the maximal Thurston--Bennequin invariant unknot in the boundary of the 0--handle bounds such a disk). The Stein 1 and 2--handles can be cancelled to produce the standard picture of the 4--ball. After this cancellation the unknot becomes a Legendrian presentation of $K$ that clearly bounds a Lagrangian disk (since the knot before cancellation did). See Figure~\ref{stein} for an example of this that originally appeared in \cite{Yasui15preprint}. 
\begin{figure}[h]
\begin{overpic}
{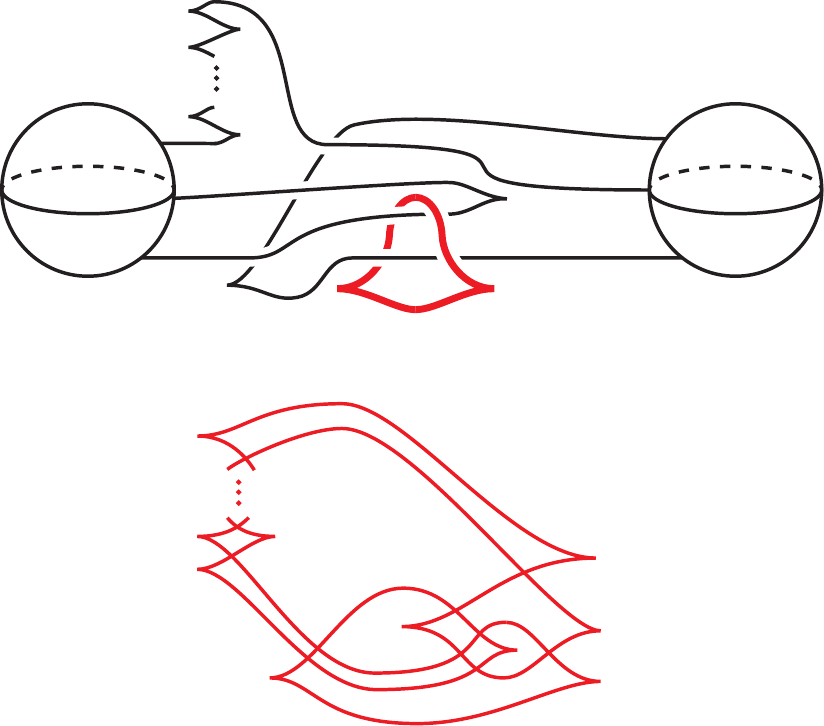}
\end{overpic}
\caption{The top figure shows Stein manifold presentation of the 4-ball with a Legendrian knot in the boundary that obviously bounds a Lagrangian disk. The Legendrian knot has Thurston--Bennequin invariant $-t$, where $t+1$ is the number of left cusps in the figure. The bottom figure is the Legendrian knot after the 1, 2--handles have been cancelled. There are $t+1$ full left handed twist in the upper left side of the bottom figure.}
\label{stein}
\end{figure}  
Removing a neighborhood of the Lagrangian disk that $K$ bounds results in a Stein manifold with Kirby diagram obtained from the Stein presentation of the 4--ball by turning the unknot into a 1--handle (that is putting a dot on the unknot). Notice that the boundary of this Stein manifold is the result of contact $(+1)$--surgery on $K$ (this is smooth 0--surgery on $K$).

The above construction yields many Legendrian knots bounding Lagrangian disks. In fact, it generates all such knots bounding regular Lagrangian disks. 
\begin{theorem}\label{construction}
A Legendrian knot $L$ in $(S^3,\xi_{std})$ bounds a regular Lagrangian disk in $(B^4,\omega_{std})$ if and only if it comes from the above construction. 
\end{theorem}

\begin{remark}
We end by noting that for certain Legendrian knots $L$ bounding Lagrangian disks constructed as above, Yasui shows how to create a Lagrangian representative of the $(n,-1)$--cable of $L$ that also bounds a Lagrangian disk, for some positive $n$. This is quite interesting as our construction in Proposition~\ref{cable} will produce a Lagrangian disk with boundary the $(n,1)$-cable of any Legendrian knot $L$ bounding a Lagrangian disk. We do not see any relation between these two constructions, nor a simple geometric construction of the Lagrangian disks coming from Yasui's construction. 
\end{remark}

\subsection{Larger contact surgeries}
Having completely answered the question of when contact $(r)$--surgery on $L\subset (S^3,\xi_{std})$ is symplectically fillable for $r \leq 1$, it is natural to ask what happens for $r>1$. 

We begin by noticing that building on work of Ghiggini, Lisca, and Stipsicz \cite{GhigginiLiscaStipsicz07} and Owens and Strle \cite{OwensStrle12} in certain cases, one can get symplectically fillable contact structures for sufficiently large $r$. 
\begin{proposition}\label{torusfill}
For a positive $(p,q)$--torus knot with maximal Thurston--Bennequin invariant, contact $(r)$--surgery will give a Stein fillable contact structure if $r\geq p+q-1$.

Additionally, for the $(2,2n+1)$--torus knot with maximal Thurston--Bennequin invariant, contact $(r)$--surgery is tight for all $r > 0$, and is symplectically (and Stein) fillable if and only if $r\geq 2n+1$. 
\end{proposition}

In \cite{LiscaStipsicz04, OwensStrle12}, it is shown that the result of smooth $4n$--surgery on the $(2,2n+1)$--torus knot admits Stein fillable contact structures, but they are not necessarily the ones arising from contact surgery on Legendrian knots in $(S^3, \xi_{std})$. The heart of the proof of Proposition~\ref{torusfill} will be to show that the Stein fillable contact structure on smooth $4n$--surgery on the $(2,2n+1)$--torus knot  is indeed the one coming from contact $(2n+1)$--surgery on a Legendrian $(2,2n+1)$--torus knot  in $(S^3, \xi_{std})$.

We now look at topological obstructions to the symplectic fillability of contact $(r)$--surgery on $L$.

\begin{theorem}\label{blowupdisk}
Let $L$ be a Legendrian knot in the standard contact 3--sphere $(S^3, \xi_{std}).$ If contact $(r)$--surgery is symplectically fillable for some $r>0$, then the transverse push-off of $L$ bounds a symplectic disk in $B^4$ blown-up some number of times.
\end{theorem}

We note that in fact any smooth knot bounds a smooth disk in $B^4$ blown-up some number of times, and this geometric observation can be turned into an effective lower bound on the framing of surgeries that result in a symplectically fillable contact structure. To state the result, we first define a function $f:\N\to \Z$ by letting $f(t)$ be the minimum of $\sum d_i^2$ over finite tuples $(d_1, \ldots, d_m)$ of non-negative integers satisfying $$\sum \left(d_i^2 - d_i\right) \geq 2t.$$  We do not have a closed-form description of $f$, but one can work out that $f$ takes the values $0$, $4$, $8$, $9$, $13$, and $16$ for $t = 0,1, 2, 3, 4,$ and $5$ respectively.  It is not hard to see that if we allow real values for $d_i$, then the minimum of $\sum d_i^2$ occurs when $m=1$ and $d_1 = d$, and satisfies $d^2 - d = 2t$. Thus, 
\[
f(t) \geq 2t + \lceil d \rceil = 2t + \left\lceil\frac{\sqrt{8t+1}+1}2\right\rceil.
\]

\begin{theorem}\label{taubound}
Let $L$ be a Legendrian knot in the standard contact 3--sphere $(S^3, \xi_{std})$ and let $\tau(L)$ be the Heegaard Floer $tau$-invariant.  If $\tau(L)\geq 0$ then contact $(r)$--surgery on $L$ is not symplectically fillable for any $r\leq f(\tau(L))-tb(L)-1$.
\end{theorem}

Note that we do not claim that contact $(r)$--surgery is symplectically fillable when $r > f(\tau) - tb(L) - 1$.  This obstruction is best seen as a smooth one, since $tb(L) + r$ is the smooth surgery coefficient corresponding to contact $(r)$--surgery on $L$.

We notice that although for many Legendrian knots, even stabilized ones, sufficiently large contact surgeries are tight, there are also knot types with no tight positive contact surgery on any Legendrian representative, such as the figure-eight knot \cite{Conway16preprint}. 

The above observations lead us to the following questions.
\begin{question}
Given a Legendrian knot $L$ in $(S^3,\xi_{std})$ for which some positive contact surgeries are tight, are sufficiently large positive contact surgeries fillable?
\end{question}

\begin{question}\label{satisfytau}
Does there exist for each $\tau \in \N$ a Legendrian knot $L$ in $(S^3, \xi_{std})$ with $\tau(L) = \tau$ such that contact $(r)$--surgery on $L$ is strongly symplectically fillable for $r \geq f(\tau) - tb(L)$?  What about $r > f(\tau) - tb(L) - 1$?
\end{question}

Proposition~\ref{torusfill} implies that the unknot, the $(2,3)$--torus knot, and the $(2, 5)$--torus knot give a positive answer to Question~\ref{satisfytau} for $\tau = 0$, $1$, and $2$ for $r \geq f(\tau) - tb(L)$.

It is also interesting to consider positive contact surgeries in other contact manifolds.
\begin{question}
Given a tight contact manifold, what criteria will determine whether contact $(r)$--surgery for $r > 0$ is fillable?
\end{question}

We note that this question for $r=1$ seems related to the classification of fillings of a contact manifold. Specifically, our proof of Theorem~\ref{thm1} clearly shows that the result of contact $(1)$--surgery on the Legendrian $L$ in $(M,\xi)$ will be strongly symplectically fillable if and only if $L$ bounds a Lagrangian disk in some symplectic filling of $(M,\xi)$. To get a more precise statement one needs to know something about the fillings of $(M,\xi)$. 

For example, it is known that all the contact structures on $L(p,1)$ have a unique Stein filling, given by a disk bundle $X_{-p}$ over $S^2$ with Euler number $-p$, except $L(4,1)$ with the universally tight  structure has two fillings, one by the disk bundle $D_{-4}$ and one by a rational homology ball $B_4$, \cite{McDuff90, PlamenevskayaVanHornMorris10}. It is also well-known that all these contact structures are supported by planar open books, {\em cf} \cite{Etnyre04b}. A result of Wendl \cite{Wendl10} says that any minimal symplectic filling of a contact structure supported by a planar open book can be deformed into a Stein filling. The arguments in the proof of Theorem~\ref{thm1} yield the following result. 
\begin{theorem}\label{lens}
Let $L$ be a Legendrian knot in a lens space $L(p,1)$ with tight contact structure $\xi$.  When $\xi$ is not the universally tight contact structure on $L(4,1)$ then the result of contact $(+1)$--surgery on $L$ bounds a symplectic manifold if and only if $L$ bounds a Lagrangian disk in the disk bundle $X_{-p}$ (with the unique Stein structure filling $(L(p,1),\xi)$). When $\xi$ is the universally tight contact structure on $L(4,1)$ then the result of contact $(+1)$--surgery on $L$ bounds a symplectic manifold if and only if $L$ bounds a Lagrangian disk in the disk bundle $X_{-p}$ (with the unique Stein structure filling $(L(p,1),\xi)$) or the rational homology ball $B_{-4}$.
\end{theorem}
Some of the other results above extend to this case too. 

\medskip
{\bf Acknowledgements:} The authors are grateful to Ko Honda for asking the question that prompted this paper and to the Georgia International Topology Conference for providing a stimulating environment. We would also like to thank Oleg Lazarev for pointing out subtleties in the relation between regular and decomposable Lagrangian submanifolds, and Roger Casals and Lenny Ng for helping us better understand these subtleties. We also thank the two anonymous referees for providing valuable comments that improved the paper. The first author was partially supported by NSF grant DMS-1344991; the second author was partially supported by NSF grant DMS-1608684.

\section{Background}
In this section we review contact surgery on Legendrian knots, Weinstein handle attachments, and admissible surgery on transverse knots. 

\subsection{Contact surgery}\label{surgery}
We refer the reader to \cite{DingGeigesStipsicz04} for details on contact surgery and \cite{Etnyre05, EtnyreHonda01b} for details on Legendrian knots, but briefly recall the relevant features here. 

A Legendrian knot $L$ in a contact 3--manifold $(M,\xi)$ has a standard neighborhood $N$ that is a solid torus with convex boundary. The dividing curves on the boundary of $N$ consist of two curves of slope $tb(L)$. Using Giroux realization, the characteristic foliation on $\partial N$ can be assumed to consist of two lines of singularities (called Legendrian divides) parallel to the dividing curves, and the rest of the foliation is non-singular and consists of curves of any slope but the dividing slope (these are called ruling curves). We note that each Legendrian divide is Legendrian isotopic to $L$ in $N$. Moreover, $\partial N$ has a neighborhood in $N$ of the form $(\partial N)\times [1/2,1]$, with $\partial N=(\partial N)\times\{1\}$ so that the contact structure is invariant in the $[1/2,1]$ direction. Thus we have an interval's worth of copies of the Legendrian divide. Any one of these will be called a \dfn{Legendrian push-off} of $L$. 

Given the above set-up, we can remove $N$ from $M$ and then glue in a solid torus $S^1\times D^2$ so that the meridian is mapped to a curve of slope $1/n$ on $\partial \overline{(M-N)}$.  Here we are measuring slopes with respect to the framing coming from the contact planes (so with respect to a Seifert framing, this slope is $tb(L)+1/n$). By \cite{Giroux00, Honda00a} there is a unique tight contact structure on $S^1\times D^2$ that extends the contact structure $\xi_{\overline{M-N}}$. We call the resulting contact structure the result of \dfn{contact $(1/n)$--surgery on $L$}.

It is well-known and easy to show that contact $(1/n)$--surgery on $L$ for $\pm n>1$ is equivalent to contact $(\pm 1)$--surgery on $|n|$ parallel push-offs of $L$. Moreover, the result of contact $(\pm 1)$--surgery on a Legendrian push-off of $L$ cancels the result of contact $(\mp 1)$--surgery on $L$.

We can also perform contact $(r)$--surgery on $L$ for any $r \neq 0$, but when $r\not=1/n$, there is not a unique choice for the contact structure on the resulting manifold (there are choices for the extension of $\xi_{M-N}$ over the surgery solid torus). In these cases, we will always consider the tight contact structure on the torus with maximally negative relative Euler class. This choice agrees with those studied in the Heegaard Floer literature \cite{Golla15, LiscaStipsicz04, MarkTosunPre} and corresponds to \dfn{inadmissible transverse surgery} \cite{Conway14?}. In the language of \cite{DingGeigesStipsicz04} we choose all stabilizations of the Legendrian knots to be negative.  For example, contact $(n)$--surgery on $L$, for an integer $n>1$, is achieved by taking a push-off of $L$ and stabilizing it once negatively to get $L'$, then do contact $(+1)$--surgery on $L$, and $(-1)$--surgery on $L'$ and on $n-2$ Legendrian push-offs of $L'$. 

\subsection{Weinstein handles} \label{whandle}
In several of our constructions below, we will need a very careful description of Weinstein handles. Details can be found in the original paper \cite{Weinstein91} or the recent book \cite{CieliebakEliashberg12}. We will restrict our attention to the 4--dimensional setting.

Models for the handles will all be contained in $\R^4$ with the symplectic structure $\omega=dx_1\wedge dy_1+ dx_2\wedge dy_2$.

{\em Weinstein 1--handles:} Consider the hypersurfaces 
\[
S_1=\{x_1^2+x^2_2+y_1^2-y_2^2=-1\}\text{ and } S_2=\{x_1^2+x^2_2+y_1^2-\delta y_2^2=\epsilon\}
\]
where $\epsilon$ and $\delta$ are positive.
The Weinstein 1--handle $H_1$ is the closure of the region of $\R^4-(S_1\cup S_2)$ that contains the origin. The core of the handle is $C=H_1\cap \{x_1=x_2=y_1=0\}$. It should be clear that $C$ is an isotropic submanifold. The attaching sphere for the handle will be $a=C\cap S_1$. We also set $A=S_1\cap H_1$, and call it the attaching region. Notice that by an appropriate choice of $\epsilon$ and $\delta$, we can arrange for $A$ to be contained in any given neighborhood of $a$ in $S_1$. 

The vector field $v=\frac 12 x_1\partial_{x_1}+2 x_2\partial_{x_2}+\frac 12 y_1\partial_{y_1}-y_2\partial_{y_2}$  is a Liouville vector field for $\omega$.  Since it is transverse to both $S_1$ and $S_2$ --- pointing into $H_1$ along $A$ and out of $H_1$ on the rest of the boundary --- there is an induced contact structure on $A$ and $\overline{(\partial H_1) - A}$. 

\begin{figure}[h]
\begin{overpic}
{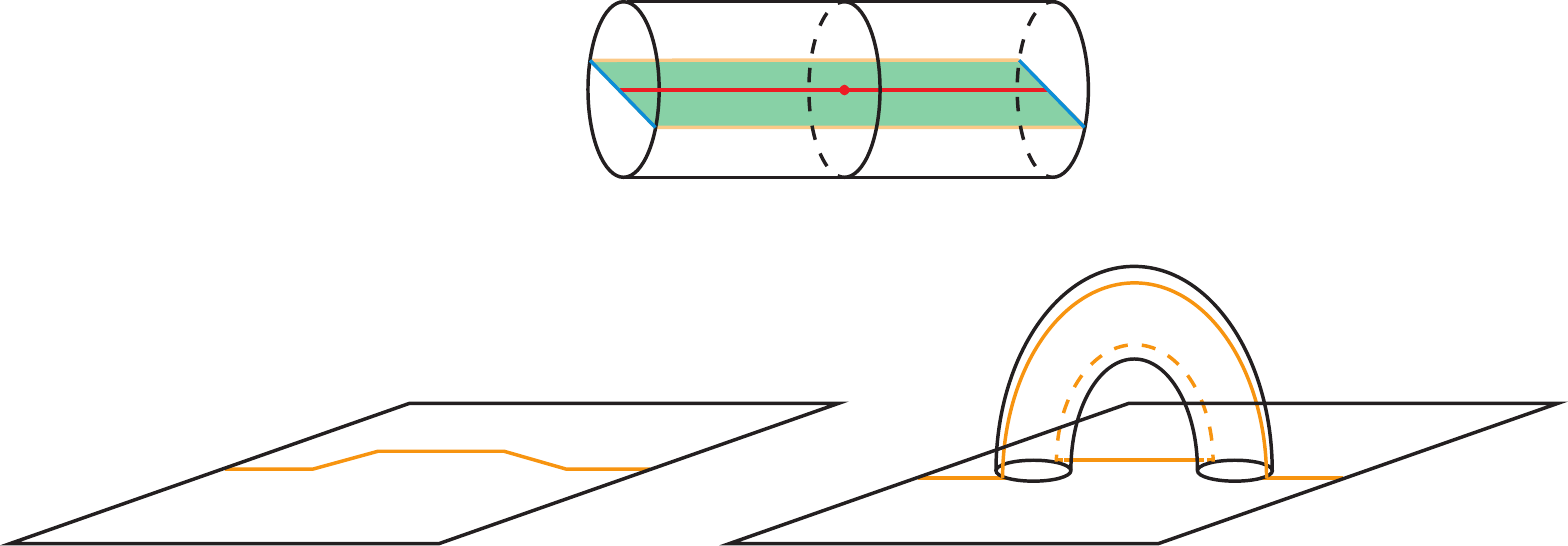}
\end{overpic}
\caption{Schematic view of attaching a Weinstein 1--handle. The top figure shows the 1--handle $D^1\times D^3$ (we can only show $D^2$ instead of $D^3$). The core $C$ is shown in red and the Lagrangian $\Lambda$ is shown in green. The Liouville field is pointing in on the right and left face of the handle and out on the other face; it is also tangent to $\Lambda$. On the bottom left is the boundary of $X$ with the Legendrian knot $L$ shown in orange. On the right the 1--handle has been attached and the Legendrian link $L'$ is shown in orange. If $L$ bounded a Lagrangian $\Lambda_L$ in $X$ then $L'$ bounds a Lagrangian obtained from $\Lambda_L$ by attaching the 2-dimensional 1--handle $\Lambda$.}
\label{decomp}
\end{figure}  

Given a symplectic 4--manifold $(X,\omega)$ with convex boundary and two points in $\partial X$ we can find a contactomorphism from $A$ to a neighborhood of the points and use this to attach the handle $H_1$ to $X$, resulting in a new symplectic manifold with convex boundary. We will need a relative version of this handle attachment as well. To this end notice that $\Lambda=\{x_1=x_2=0\}\cap H_1$ is a Lagrangian disk and $A'=\Lambda\cap A$ is the union of two Legendrian arcs, as is $B'=\Lambda\cap ((\partial H_1)-A)$. Given a Legendrian link $L$ in $\partial X$ bounding a Lagrangian submanifold $\Lambda_L$ in $X$ that is transverse to $\partial X$, and given any two points on $L$, there is a contactomorphism from $A$ to a neighborhood of these points taking $A'$ to arcs on $L$ containing the points. We can use this contactomorphism to attach a Weinstein 1--handle to $X$ to obtain $X'$ and simultaneously attach the Lagrangian 1--handle $\Lambda$ to $\Lambda_L$. The boundary of this new Lagrangian in $\partial X'$ is a Legendrian link $L'$ obtained from $L$ by surgery on the two points.

{\em Weinstein 2--handles:} Consider the hypersurfaces 
\[
S_1=\{x_1^2+x^2_2-y_1^2-y_2^2=-1\}\text{ and } S_2=\{x_1^2+x^2_2-\delta y_1^2-\delta y_2^2=\epsilon\}
\]
where $\epsilon$ and $\delta$ are positive.
The Weinstein 2--handle $H_2$ is the closure of the region of $\R^4-(S_1\cup S_2)$ that contains the origin. The core of the handle is $C=H_2\cap \{x_1=x_2=0\}$ and the co-core is $K=H_2\cap \{y_1=y_2=0\}$. It should be clear that $C$ and $K$ are Lagrangian submanifolds. The attaching sphere for the handle will be $a=C\cap S_1$. We also set $A=S_1\cap H_1$, and call it the attaching region. Notice that by an appropriate choice of $\epsilon$ and $\delta$, we can arrange for $A$ to be contained in any given neighborhood of $a$ in $S_1$. 

The vector field $v=2 x_1\partial_{x_1}+2 x_2\partial_{x_2}- y_1\partial_{y_1}-y_2\partial_{y_2}$  is a Liouville field for $\omega$.  Since it is transverse to both $S_1$ and $S_2$ --- pointing into $H_1$ along $A$ and out of $H_1$ on the rest of the boundary --- there is an induced contact structure on $A$ and $\overline{(\partial H_1) - A}$. 

Given a symplectic 4--manifold $(X,\omega)$ with convex boundary and a Legendrian knot $L$ in $\partial X$, we can find a contactomorphism from $A$ to a neighborhood of $L$, and use this contactomorphism to attach the 2--handle $H_2$ to $X$ to get a new symplectic manifold $X'$ with convex boundary. It is well-known and easy to check that $\partial X'$ is obtained from $\partial X$ by contact $(-1)$--surgery on $K$, \cite{Gompf98}.

We note that the boundary of the co-core $\partial K$ is a Legendrian knot in $\partial X'$ and is isotopic to a Legendrian push-off of $L$. Moreover, inside any open set containing $K$ in $X'$, one may find a neighborhood of $K$ that, when removed from $X'$, results in a symplectic manifold $X''$ with convex boundary. Clearly $X''$ is diffeomorphic to $X$ (although the symplectic structure has been deformed). It is well-known and easy to check that $\partial X''$ is obtained from $\partial X'$ by contact $(+1)$--surgery on $\partial K$. 

To prove our main theorem about $(+1)$--contact surgeries we need the following simple observation (this result originally appeared in \cite{DingLiPreprint}, attributed to the second author of this paper). 
\begin{lemma}\label{minimal}
Let $(X,\omega)$ be a strong symplectic filling of the contact manifold $(M,\xi)$ and $(X',\omega')$ the result of attaching a Weinstein 2--handle. If $(X,\omega)$ is minimal, then so is $(X',\omega')$. 
\end{lemma}
\begin{proof}
Since $(X,\omega)$ has a strongly convex boundary, there is a neighborhood $N$ of $M$ in $X$ that is symplectomorphic to a piece of the symplectization of $(M,\xi)$. The Weinstein 2--handle is attached to $N$ to create a manifold $N'\subset X'$. One can choose an almost complex structure on $X'$ that is compatible with $\omega'$ and is a complex structure on $N'$, that is, $N'$ can be taken to be a Stein cobordism, \cite{CieliebakEliashberg12}, and thus there is a pluri-subharmonic function $\phi:N'\to \R$ with the boundary components of $N'$ being level sets of $\phi$. 

Now assume that $(X',\omega')$ is not minimal. Then there is a pseudo-holomorphic sphere $S$ in $X'$, \cite{Taubes96}. If $S$ is disjoint from $N'$ then it lies entirely in $X$, which contradicts the minimality of $(X,\omega)$. Thus $S$ must non-trivially intersect $N'$. We may assume that $S$ intersects the lower boundary of $N'$ transversely, 
so $S\cap N'$ is a pseudo-holomorphic embedding in $N'$ with boundary in the lower boundary component of $N'$. Thus, there will be a maximum of $\phi$ on $S$ at an interior point of $S\cap N'$. This contradicts the fact that $\phi$ is pluri-subharmonic. Thus $(X', \omega')$ is minimal. 
\end{proof}

\subsection{Transverse surgery}\label{transsurg}
As an auxiliary tool in our study of surgeries on Legendrian knots we need to consider surgeries on transverse knots. We briefly recall admissible transverse surgery here, see \cite{BaldwinEtnyre13} for more details.

Let $T$ be a transverse knot in a contact 3--manifold $(M,\xi)$, and fix a framing of $T$ against which we will measure all our slopes. There is a neighborhood $N$ of $T$ in $M$ and some $R<\pi$ such that $N$ is contactomorphic to $S^1\times D^2_R$ with the contact structure $\xi_{cyl}=\ker (\cos r\, d\phi + r \sin r\, d\theta)$, where $D^2_R$ is a disk of radius $R$, and where the fixed framing is taken to the product structure on $S^1\times D^2_R$.  If $\frac pq<\frac{\cos R}{R\sin R}$, then we call $\frac pq$ an \dfn{admissible slope}, and there is a unique radius $r_0$ such that the torus $T_{p/q}=\{r=r_0\}$ in $S^1\times D^2$ has characteristic foliation of curves of slope $\frac pq$ (that is, $p$ meridians and $q$ longitudes, using the fixed framing). Removing the interior of the solid torus that $T_{p/q}$ bounds and collapsing the leaves of the characteristic foliation on $T_{p/q}$ (technically this is performing a contact cut) will result in a manifold $M_{p/q}$ obtained from $M$ by $\frac pq$ surgery on $T$ that supports a contact structure $\xi_{p/q}$ such that $\xi_{p/q}$ agrees with $\xi$ on the complement of $T'$, which is the image of $T_{p/q}$ in $M_{p/q}$. Notice that $T'$ is a transverse knot in $(M_{p/q},\xi_{p/q})$. We say $(M_{p/q},\xi_{p/q})$ is obtained from $(M,\xi)$ by \dfn{admissible transverse surgery on $T$ with slope $\frac pq$}.

In \cite{Gay02a}, David Gay showed how to attach symplectic 2--handles along transverse knots. 
\begin{theorem}[Gay 2002, \cite{Gay02a}]\label{gaythm}
Let $(X,\omega)$ be a symplectic manifold with strongly convex boundary. Let $T$ be a transverse knot in $\partial X$, and let $F$ be a framing of $T$.  If $0$ is an admissible slope (when using $F$ as the fixed framing), then there is a symplectic structure on the manifold $Y$ obtained from $X$ by attaching a 2--handle to $T$ with framing $0$, that has weakly convex boundary, and the contact structure on $\partial Y$ is obtained from the contact structure on $\partial X$ by admissible transverse surgery on $T$ with slope $0$.
\end{theorem}

In \cite[Theorem~5]{Wendl2013}, Chris Wendl showed that the symplectic structure on $Y$ can be chosen so that the co-core of the 2--handle is a symplectic disk.

When discussing slopes on tori it will be useful to consider the Farey tessellation, see \cite{Honda00a} for more details. Recall this is a tessellation of the unit disk $D^2$ formed as follows. Label $(0,1)$ by $0$, $(0,-1)$ by  $\infty$, and draw a hyperbolic geodesic connecting them. Then for a point on $\partial D^2$ with positive $x$-coordinate that is half-way between two points that are already labeled with $p/q$ and $p'/q'$, we label that point by $(p+p')/(q+q')$ and connect it to the two points by hyperbolic geodesics (here $0=0/1$ and $\infty= 1/0$). Do the same for points with negative $x$-coordinate by considering $\infty=-1/0$. All the rational numbers will show up labeling some point, they will appear in order, and two will be connected by a geodesic if and only if the vectors corresponding to the slopes form a basis for $\Z^2$ (here the vector corresponding to $p/q$ has coordinates $p$ and $q$). Given a torus $T^2$ and a basis for $H_1(T^2)$, the slopes correspond to elements of the first homology. Moreover, if we have chosen a basis for $H_1(S^1\times D^2)$ so that the meridian has slope $r/s$ then the longitudes for the torus will be exactly those slopes that have an edge to $r/s$ in the Farey tessellation. 

It will also be useful to know how to relate neighborhoods of transverse and Legendrian knots. Let $L$ be a Legendrian knot and $N$ a standard neighborhood of $L$ with convex boundary having dividing curves giving the framing $F$.  We will label all slopes using a basis for $H_1(T^2)$ coming from the framing on $N$ --- $0$ corresponds to $F$, and $\infty$ corresponds to the meridional slope of $L$ --- and slopes are represented on $S^1$ as in the Farey tessellation. There is a transverse push-off $L^t$ of $L$ in $N$, see \cite{EtnyreHonda01b}, and $L^t$ will have a neighborhood as above realizing all slopes less than $0$ as characteristic foliations.  Moreover, if one performs contact $(r)$--surgery on $L$ in $N$, then there is a transverse knot $T'$ in the new solid torus $N'$, and the slopes realized by characteristic foliations on boundary-parallel tori in $N'$ have slopes clockwise of $r$ and counterclockwise of $0$.  If we then do admissible transverse surgery on $T'$ with slope $s \in (r, \infty)$, then the resulting contact manifold is the same as the result of contact $s$--surgery on $L$, see \cite{Conway14?}.

\begin{remark} 
Recall that at the end of Section~\ref{surgery}, we specified that in this paper, the contact structure on the result of contact $(r)$--surgery on a Legendrian knot always refers to the contact structure on the surgered manifold which when restricted to the surgery torus has maximally negative relative Euler class. The above properties of $T'$ are true only for this choice.
\end{remark}

Instead of doing admissible transverse surgery on $T'$, Baldwin and Etnyre showed in \cite{BaldwinEtnyre13} that for certain $s$, the same result can achieved via Legendrian surgery.
\begin{theorem}[Baldwin--Etnyre, \cite{BaldwinEtnyre13}]\label{BEtheorem}
Let $T' \subset N'$ come from contact $(r)$--surgery on $L$, with notation as above.  Let $s_0$ be the slope the furthest clockwise of $r$ such that $s_0$ is counter-clockwise of $0$ and has an edge to $r$ in the Farey tessellation.  Then for any $s < s_0$, there exists a Legendrian link in $N'$ such that the result of Legendrian surgery on each component of the link is the same as the result of contact $(s)$--surgery on $L$.
\end{theorem}

\section{Lagrangian fillings}\label{fillings}
In this section we begin by constructing standard neighborhoods of Lagrangian disks in symplectic manifolds, and then turn to a discussion of decomposable and regular Lagrangian submanifolds. 

\subsection{Neighborhoods of Lagrangian disks} We will need to understand the effect of removing Lagrangian disks from a symplectic manifold.  Once a neighborhood of the boundary of a disk is normalized, see Lemma~\ref{mt}, the main result along these lines follows from a standard Moser type argument. In particular, arguments similar to this can be found in Section~12 of \cite{CieliebakEliashberg12}. 
\begin{theorem}\label{remove}
Let $(X,\omega)$ be a symplectic manifold with convex boundary. If $D$ is a Lagrangian disk properly embedded in $X$ and transverse to the boundary, then there is a neighborhood $N$ of $D$ in $X$ such that $(X'=\overline{X-N}, \omega'=\omega|_{X'})$ is a symplectic manifold with convex boundary and the contact manifold $\partial X'$ is obtained from the contact manifold $\partial X$ by contact $(+1)$--surgery on $\partial D$. Moreover, if $X$ had strongly convex boundary then so does $X'$.
\end{theorem}

We will need two preliminary lemmas. 
\begin{lemma}[Etnyre, Lidman, and Ng 2017, \cite{EtnyreLidmanNgPre} ]\label{mt}
Let $(X,\omega)$ be a symplectic manifold with convex boundary and $\Lambda$ a Lagrangian submanifold properly embedded in $X$ that is transverse to $\partial X$. Then there is a Liouville vector field $v$ defined near $\partial \Lambda$ that is tangent to $\Lambda$. Moreover, if $(X,\omega)$ is a strong filling of $\partial X$ then $v$ may be extended to a Liouville vector field near $\partial X$. 
\hfill\qed
\end{lemma}

\begin{lemma}\label{nbhd}
Let $(X,\omega)$ be a symplectic manifold with convex boundary and $D$ a Lagrangian disk in $X$ that is transverse to $\partial X$. Then $D$ has a neighborhood $N$ symplectomorphic to a neighborhood of the co-core of a Weinstein 2--handle. Moreover, if $X$ has strongly convex boundary then we can assume the Liouville vector field for $X$ near $\partial X$ and that on the 2--handle agree near $\partial D$.  
\end{lemma}

\begin{proof}[Proof of Theorem~\ref{remove}]
In Section~\ref{whandle} we discussed the effect of removing the co-core of a Weinstein 2--handle. This discussion caries over to a neighborhood of a general Lagrangian disk by Lemma~\ref{nbhd}.
\end{proof}

\begin{proof}[Proof of Lemma~\ref{nbhd}]
Let $(H_2,\omega_2)$ be a Weinstein 2--handle as discussed in Section~\ref{whandle}; we will use notation from that section for the co-core $K$ and other parts of the handle. 
There is an expanding vector field $v$ for a neighborhood of $\partial D$ in $\partial X$, and let $v_2$ be the expanding vector field for $H_2$ discussed in Section~\ref{whandle}. Using Lemma~\ref{mt} we can assume that the $v$ is  tangent to $D$ near $\partial D$. The 1--form $\alpha_1=\iota_{v}\omega$ is a contact form for the contact structure on $\partial X$ near $\partial D$, and $\alpha_2=\iota_{v_2}\omega_2$ is a contact form for a neighborhood of the boundary of the co-core in $\partial H_2$. 

We will construct our symplectomorphism in three steps. 

\noindent
{\bf Step 1:} {\em Construct a neighborhood $U$ of $\partial D$ in $\partial X$, a neighborhood $U'$ of $\partial K$ in $\partial H_2$, and a diffeomorphism $f:U'\to U$ such that $f^*\alpha_1=\alpha_2$ and $f(\partial K)=\partial D$.}

There is an annulus $A$ containing $\partial D$ that is tangent to $\xi_1 = \ker \alpha_1$ along $\partial D$, transverse to $\xi_1$ away from $\partial D$, and transverse to the Reeb vector field for $\alpha_1$. Notice that $d\alpha_1$ is a symplectic form when restricted to $A$ and $\partial D$ is a Lagrangian submanifold in $A$. We can construct a similar annulus $A'$ for $\partial K$. Since neighborhoods of Lagrangian submanifolds are all standard we can assume the annuli are chosen so there is a symplectomorphism $h:A'\to A$ taking $\partial K$ to $\partial D$. 

Since the Reeb flow of $\alpha_i$ preserves $d\alpha_i$, we can use the Reeb flows to extend $h$ to a diffeomorphism from a neighborhood $W'$ of $\partial K$ in $\partial H_2$ to a neighborhood $W$ of $\partial D$ in $\partial X$. Clearly $h^* d\alpha_1= d\alpha_2$ and, moreover, $h^*\alpha_1=\alpha_2$ along $\partial K$, since the 1-forms have the same kernel and agree on the Reeb field. 

We now construct a diffeomorphism $g$ from a neighborhood $U'$ of $\partial K$ to another such neighborhood (both contained in $W'$) such that $g^*(h^*\alpha_1)=\alpha_2$. Then setting $f$ equal to $h\circ g$ and $U=f(U')$  will give the claimed contactomorphism. To find $g$, consider the 1-forms
\[
\beta_s=\alpha_2 + s(h^*\alpha_1-\alpha_2),
\]
for $s\in[0,1]$.
On a small enough neighborhood of $\partial K$, the $\beta_s$ are all contact forms. This is because the contact condition is an open condition, the $\beta_s$ all agree along $\partial K$, and $d\beta_s$ is a symplectic form on $\ker \beta_s$ for all $s$. 

Notice that $d(h^*\alpha_1-\alpha_2)=0$ everywhere and $h^*\alpha_1-\alpha_2=0$ along $\partial {K}$. Thus, there is a function $\phi$ from a neighborhood of $\partial {K}$ to $\R$ such that $d\phi= -(h^*\alpha_1-\alpha_2)$ and $\phi=0$ along $\partial {K}$. 

Denote the Reeb vector field of $\beta_s$ by $w_s$ and set $u_s=\phi w_s$. Let $\psi_s$ denote the flow of $u_s$, and let $U'$ be a neighborhood of $\partial {K}$ on which the flow is defined up to time $s=1$ (since $u_s$ is 0 along $\partial {K}$ there is such a neighborhood). One can compute $\frac{d}{ds} \left(( \psi_s)^*\beta_s \right) =0$.
Thus we see that $(\psi_s)^*\beta_s=\beta_0$. In particular, setting $g=\psi_1$ gives the claimed diffeomorphism. 

\noindent
{\bf Step 2:} {\em Construct neighborhoods $V$ of $\partial D$ in $X$ and $V'$ of $\partial K$ in $H_2$  and a diffeomorphism $F:V'\to V$ such that $U=V\cap \partial X$, $U'=V'\cap \partial H_2$, $F$ is an extension of $f$, and $F$ is a symplectomorphism.} 

The flow of the $v_1$ and $v_2$ in $X$ and $H_2$, respectively,  can be used to extend the neighborhoods $U$ and $U'$, respectively, to neighborhoods $V$ and $V'$. Moreover, the flows allow us to extend $f$ to a diffeomorphism $F$ with the required properties. 

\noindent
{\bf Step 3:} {\em Extend $F$ to a symplectomorphism $F'$ of a neighborhood $N'$ of $K$ in $H_2$ to a neighborhood of $D$ in $X$.} It is clear that $F'$ is the desired symplectomorphism. 

We can first extend $F$ across the disk $K$ (taking $K$ to $D$). Then choosing appropriate metrics on $H_2$ and $X$, we can use the exponential map to further extend this map to a diffeomorphism of a neighborhood of $K$ to a neighborhood of $D$ so that it agrees with $F$ on $V'$ and preserves the symplectic structure along $K$. It is now a standard application of Moser's technique to isotope the map, relative to $D$ and $V'$, to a symplectomorphism. 
\end{proof}

\subsection{Decomposable Lagrangian surfaces}\label{dls}
A systematic way to build Lagrangian surfaces with Legendrian boundary was given in \cite{EkholmHondaKalman16}. More specifically, a sequence of the following moves can be used to construct a Lagrangian surface in a piece of the symplectization of a contact manifold, and such a surface will be called a \dfn{decomposable surface}. The sequence of moves are:
\begin{enumerate}
\item\label{1} Birth of maximal Thurston--Bennequin unknot. 
\item\label{2} Pinch move.
\item Legendrian isotopy. 
\end{enumerate}
Moves \eqref{1} and \eqref{2} are shown in Figure~\ref{decomp} while the last move is obvious. 
\begin{figure}[h]
\begin{overpic}
{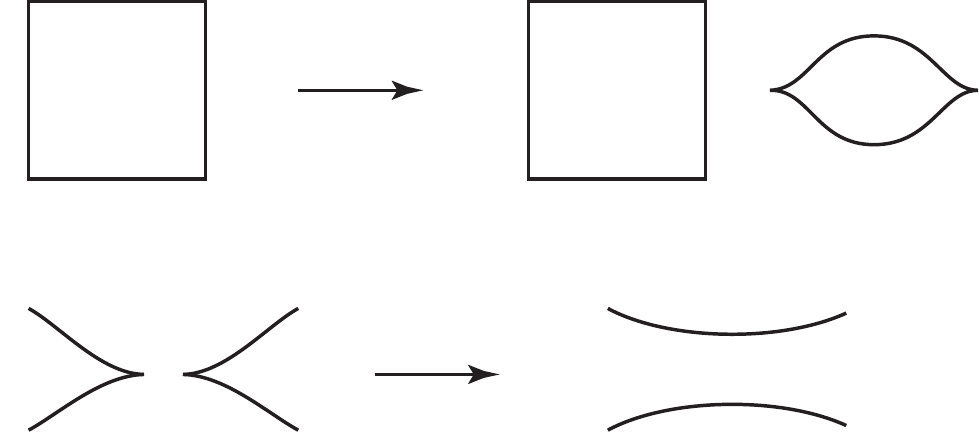}
\put(30, 97){$L$}
\put(173, 97){$L$}
\end{overpic}
\caption{The top figure shows the birth of a maximal Thurston-Bennequin unknot. The bottom figure shows a pinch move.}
\label{decomp}
\end{figure}  

To see the Lagrangian cobordisms associated to the above moves, we will build Weinstein handlebody structures on a product cobordism, and for later use, we note that there is a Liouville vector field tangent to the Lagrangian submanifolds. 

\smallskip
\noindent
{\bf Move (1):} Given a Legendrian link $L$ in $(M,\xi)$, let $\Lambda=[a,b]\times L$ in a piece $X=[a,b]\times M$ of the symplectization of $(M,\xi)$. Now take a disjoint union of $X$ and $B^4$ with its standard symplectic structure, and take the obvious Lagrangian disk in $B^4$ with boundary the unknot. Now attach a Weinstein 1--handle connecting $[a,b]\times M$ to $B^4$ so that the attaching region is disjoint from both $\Lambda$ and the Lagrangian disk. The result is a Weinstein structure on $[a,b]\times M$ that can be deformed to the standard structure, in which we see a Lagrangian cobordism in which a maximal Thurston--Bennequin unknot is born. Notice also that the Liouville vector field of the Weinstein structure is tangent to the Lagrangian cobordism. 

\smallskip
\noindent
{\bf Move (2):}  The handle addition describing the pinch move is shown in Figure~\ref{pinchmove}.
\begin{figure}[h]
\begin{overpic}
{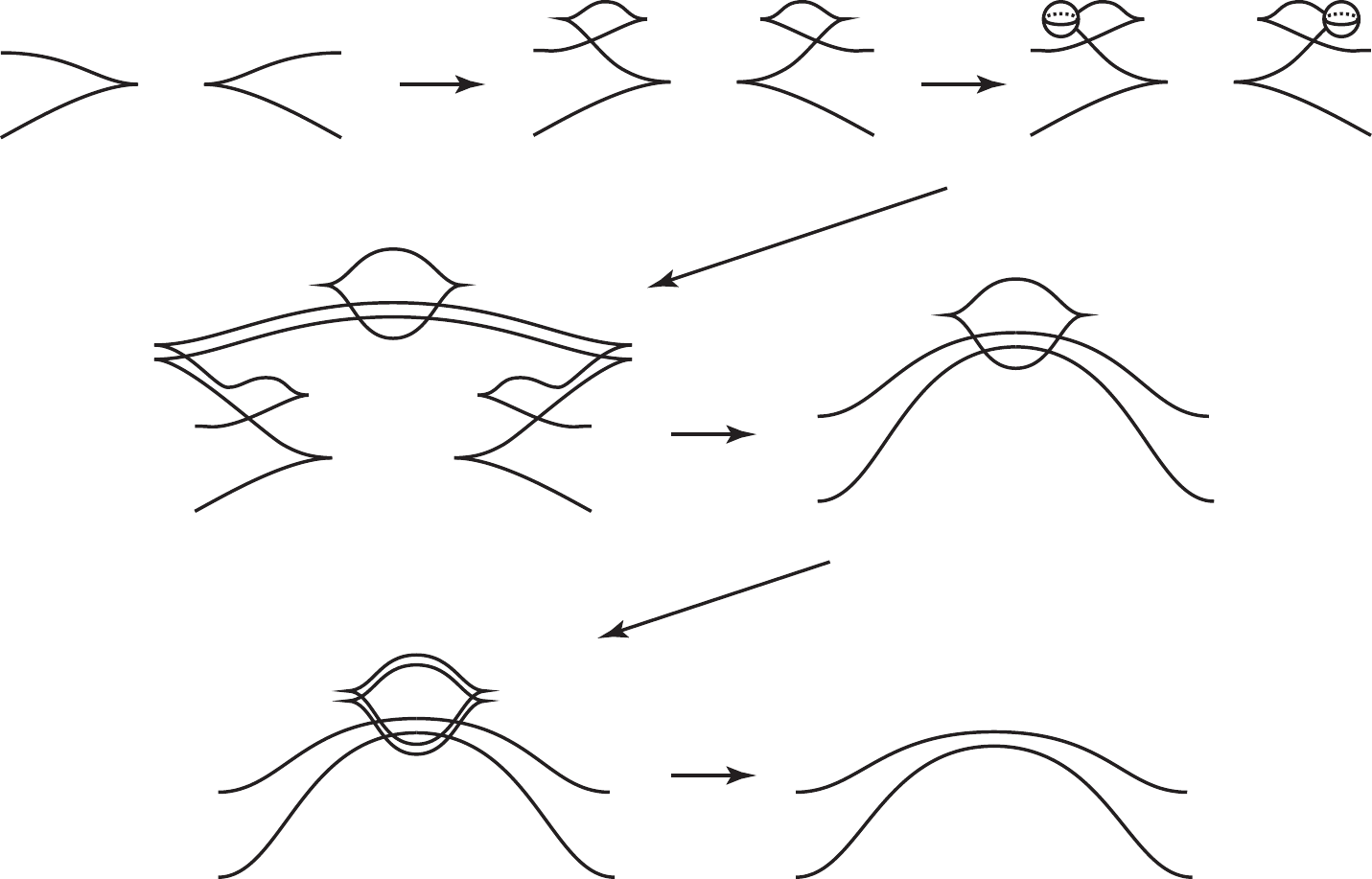}
\put(107, 153){$(+1)$}
\put(295, 145){$(+1)$}
\put(114, 27){$(+1)$}
\put(114, 72){$(-1)$}
\end{overpic}
\caption{Sequence of moves constructing a pinch move cobordism. On the upper row, we see two Reidemeister moves and then a 1--handle is attached, that is converted to a contact $(+1)$--surgered 2--handle on the second row, followed by a lightbulb move. On the third row a cancelling 2--handle is added, and then is cancelled in the final figure.}
\label{pinchmove}
\end{figure}
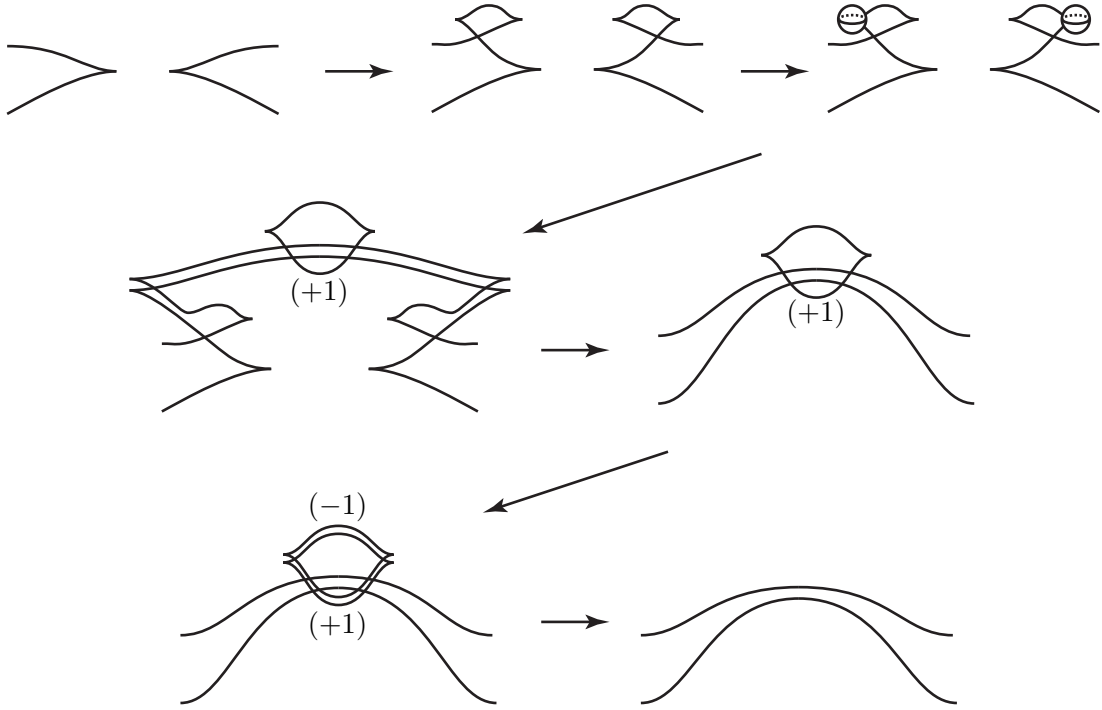  
More specifically, the pinch move applied to $L$ in the contact manifold $(M,\xi)$ is effected by attaching a cancelling pair of 1 and 2--handles to the symplectization $[a,b]\times M$ of $(M, \xi)$. To this end, we can find a Darboux ball $B$ in $M$ so that $L\cap B$ is seen in the upper left of Figure~\ref{pinchmove}.  Now a Weinstein 1--handle is attached to the symplectization in $B$. As discussed in Section~\ref{whandle}, this can be done so that a Lagrangian 1--handle is attached to $[a,b]\times L$, and there is a Liouville vector field tangent to the Lagrangian cobordism. Let the upper boundary of the Lagrangian be the Legendrian $L'$. See the upper right of Figure~\ref{pinchmove}. We now will add a canceling Weinstein 2--handle. Before we do this, we alter the representation of of the upper boundary of the symplectic cobordism as shown in the middle row of Figure~\ref{pinchmove}. Notice that the first diagram in the second row describes the same contact 3--manifold as the diagram in the upper right of the figure,  {\em cf.\@} \cite{DingGeiges09}. The next diagram is just an isotopy of $L'$, {\em cf.\@} \cite{Gompf98}. Now in the bottom left of Figure~\ref{pinchmove} the result of attaching a cancelling Weinstein 2--handle is shown. The result is a Weinstein structure on $[a,b]\times M$ that can be deformed to the standard structure \cite{CieliebakEliashberg12}, in which we see a Lagrangian cobordism in which a pinch move has been done. Notice also that the Liouville vector field of the Weinstein structure is tangent to the Lagrangian cobordism. 

\smallskip
\noindent
{\bf Move (3):} In \cite{EkholmHondaKalman16}, it was shown how to take a Legendrian isotopy from $L$ to $L'$ and create an exact Lagrangian cobordism $\Lambda$ in a piece of the symplectization $([a,b]\times M, \omega =d\lambda)$, where $\alpha$ is a contact form for our contact structure and $\lambda=e^t \alpha$ (in fact, this was only shown in \cite{EkholmHondaKalman16} for $M = S^3$, but since any isotopy can be broken into a sequence of isotopies supported in Darboux balls, their statement applies for general $M$ as well). There is a symplectomorphism $\phi$ from a neighborhood $N$ of $\Lambda$ in $[a,b]\times M$ to a neighborhood $N'$ of the product cobordism $[a,b]\times Z$ in the symplectization of the 1--jet space of $S^1$ (where $Z$ is the zero section of the jet space). Let $\beta$ be the standard contact form on the jet space, so $d(e^t \beta)$ is the symplectic form on the symplectization.  Note that $\lambda'=\phi^* (e^t \beta)$ is a primitive for $\omega$ on $N$ and $\lambda'-\lambda$ is closed. Moreover, since $\Lambda$ is exact for both $\lambda$ and $\lambda'$, and $N$ is homotopy equivalent to $\Lambda$, we see that there is a function $h:N\to \R$ such that $dh=\lambda'-\lambda$. Now if $f$ is a function on $M$ that is 1 near $\Lambda$ and 0 outside $N$ then $\lambda + d(fh)$ is a primitive for $\omega$ and equal to $\lambda'$ near $\Lambda$. Thus $(\phi^{-1})_*(\partial_t)$ is the Liouville field for $\lambda + d(fh)$ near $\Lambda$ and it is tangent to $\Lambda$. Thus, we have deformed the Weinstein structure on the symplectization as desired.

\begin{proof}[Proof of Proposition~\ref{connectsum}] The connected sum of two Legendrian knots can be achieved by a pinch move. Thus, if $L$ and $L'$ bound disjoint Lagrangian disks, we may perform a pinch move to create a Lagrangian disk bounded by $L\# L'$.
\end{proof}

\begin{proof}[Proof of Proposition~\ref{cable}]
In \cite{CornwellNgSivek16}, it was shown that if there is a Lagrangian concordance from $L$ to $L'$ then there is also a Lagrangian concordance from a given Legendrian satellite of $L$ to the same satellite of $L'$ (in that paper it is assumed the concordance is a product at the end, but this hypothesis is not necessary, in light of Lemma~\ref{mt}, see \cite{EtnyreLidmanNgPre}). 
It is easy to check that the $(n,1)$--cable of the maximal Thurston-Bennequin invariant unknot $U$ described in Theorem~\ref{cable} results in $U$. Thus if $L$ bounds a Legendrian disk, then there is a Lagrangian concordance from $U$ to $L$, and hence a Lagrangian concordance from $U$ to the $(n,1)$--cable of $L$. Hence, the $(n,1)$-cable of $L$ bounds a Lagrangian disk. 

In \cite{LiuSabloffYacavoneZhouPre} it was shown that a decomposable concordance between knots induces a decomposable concordance between certain Legendrian satellites of the knots, and in particular they construct a decomposable concordance between the $(n,1)$--cable of the unknot and the $(n,1)$--cable of $L$, if $L$ bounds a decomposable disk. Thus, the cable also bounds a decomposable Lagrangian disk. 
\end{proof}

Given a Weinstein manifold $(W,\omega)$ (that is a manifold constructed from a union of standard symplectic 4--balls by attaching Weinstein handles), notice that there is a global Liouville vector field $v$. A Lagrangian submanifold $\Lambda$ is called \dfn{regular} if the Weinstein structure is Weinstein homotopic to one in which $\Lambda$ is tangent to $v$.

From our discussion above the following lemma is clear. 
\begin{lemma}\label{regular}
If a Lagrangian submanifold in the symplectization of a contact manifold is decomposable then it is regular. \hfill \qed
\end{lemma}
\begin{remark}
 Golla and Juh{\'a}sz had also observed this lemma, but decided not to include it in \cite{GollaJuhaszPre} once this paper appeared. 
\end{remark}

Our main interest in regular Lagrangian submanifolds is due to the following result. 
\begin{theorem}[Eliashberg, Ganatra, and Lazarev 2016, \cite{EliashbergGanatraLazarevPre}]\label{cocore}
Let $(W,\omega)$ be a Weinstein manifold and $D$ a properly embedded Lagrangian disk. Then $D$ is regular if and only if there is a Weinstein handlebody structure on $W$ for which $D$ coincides with the co-core Lagrangian disk of one of the Weinstein $2$--handles if and only if there is a Weinstein handlebody structure on $W$ for which $D$ coincides with the the obvious Lagrangian disk in the 0-handle of a Weinstein structure on $(W,\omega)$.
\end{theorem}

Thus we have the following corollary.
\begin{corollary}\label{steincomp}
The symplectic structure on the complement of a Lagrangian disk in $(B^4,\omega_{std})$ coming from Theorem~\ref{remove} has a Stein structure if and only if the disk is regular. In particular, the complement of a decomposable Lagrangian disk in $(B^4,\omega_{std})$ has a Stein structure. 
\end{corollary}
\begin{proof}
If the disk is decomposable then it is regular, by Lemma~\ref{regular}, and thus by Theorem~\ref{cocore}, there is a Weinstein structure on $B^4$ for which the disk is the co-core of a Weinstein 2--handle. Thus, removing a neighborhood of the disk is the same as removing the 2--handle. So clearly the complement is a Weinstein manifold, and thus a Stein manifold \cite{CieliebakEliashberg12}.

On the other hand, if the complement of the disk has a Stein structure, then we can get back to the 4--ball by adding a 2--handle whose co-core will be isotopic to the disk. By construction, it is regular. 
\end{proof}

We end with a proof of Theorem~\ref{construction} stating that regular Lagrangian disks come from the construction of Yasui and McCullough. 
\begin{proof}[Proof of Theorem~\ref{construction}]
If a Lagrangian disk $D$ in $(B^4,\omega_{std})$ is regular, then by Theorem~\ref{cocore} there is a Weinstein structure on $(B^4,\omega_{std})$ for which $D$ is the Lagrangian equatorial disk in the 0--handle. This is precisely the set up in the construction discussed after Theorem~\ref{ribbon}.

Conversely, if $D$ comes from this construction then it is clearly regular (since it is the Lagrangian equatorial disk in the 0--handle).
\end{proof}

\section{Positive contact surgeries}
We begin by proving our main theorem concerning when contact $(r)$--surgery is fillable, for $r \in (0, 1]$. 
\begin{proof}[Proof of Theorem~\ref{thm1}]
Let $L$ be a Legendrian knot in $(S^3,\xi_{std})$ and $L'$ a parallel copy of $L$. Let $(M,\xi)$ be the result of ($+1$)--contact surgery on $L$ and assume it is strongly symplectically fillable by $(X,\omega)$. We can blow down any exceptional spheres in $(X,\omega)$, and thus assume that $(X,\omega)$ is minimal. 

Now attach a Weinstein 2--handle to $(X,\omega)$ along $L'$ (which of course sits in $(M,\xi)$) to obtain the symplectic manifold $(X',\omega')$. From Lemma~\ref{minimal} we know that $(X',\omega')$ is minimal. Moreover, it is a strong symplectic filling of $\partial X'=S^3$. A well-known result of Gromov and McDuff \cite{Gromov85, McDuff90} says that $X'$ is symplectomorphic to $B^4$ with the standard symplectic structure. From the discussion in Section~\ref{surgery}, it is clear that the boundary of the co-core of the Weinstein 2--handle is Legendrian isotopic to $L$ and thus, since the co-core is Lagrangian, we see that $L$ bounds a Lagrangian disk in the $4$--ball. 

Now assume that $L$ is a Legendrian knot in $(S^3,\xi_{std})$ that bounds a Lagrangian disk $D$ in $(B^4,\omega_{std})$. Theorem~\ref{remove} says that we can remove a neighborhood of this disk to obtain a symplectic manifold $(X,\omega)$ with strongly convex boundary that is the result of contact $(+1)$--surgery on $L$. 

We now verify that the result of removing a Lagrangian disk from $B^4$ will be an exact filling with the homology of $S^1\times D^3$. By Lemma~\ref{nbhd} there is a neighborhood of $D$ in $B^4$ symplectomorphic to a neighborhood of the co-core of a Weinstein 2--handle. Removing this neighborhood will result in a symplectic manifold $(X,\omega)$ with convex boundary. Removing a disk from $B^4$ will give a homology $S^1\times D^3$. Since the filling is a strong filling, there is a primitive $\alpha$ for $\omega$ near $\partial X$. Moreover, $\omega\in H^2(X)=0$ so there is a primitive $\beta$ for $\omega$ on all of $X$. In a neighborhood of $\partial X$, we have that $\alpha-\beta$ is a closed 1--form and so defines an element of $H^1(\partial X)=\Z$. By adding a closed form to $\beta\in H^1(X)$ if necessary, we can assume that $\alpha-\beta=0$ in $H^1(\partial X)$. Thus there is a function $h$ defined near $\partial X$ such that $dh=\alpha-\beta$. Now if $f$ is a function equal to 1 near $\partial X$, zero outside a larger neighborhood of $\partial X$ then $\beta + d(fh)$ is a global primitive for $\omega$ that agrees with $\alpha$ near $\partial X$. Thus $(X,\omega)$ is an exact symplectic filling of $\partial X$.

We have now completed the theorem for contact $(r)$--surgery when $r=1$. We are left to see that when $r\in(0,1)$, the resulting manifold is never fillable (a stronger statement than showing that it is not strongly symplectically fillable). To this end, assume contact $(r)$--surgery on $L$ in $(S^3,\xi_{std})$ is symplectically fillable by $(X,\omega)$. Since smoothly this is not $0$--surgery, we know $\partial X$ is a rational homology sphere and hence we can assume that the filling is a strong filling, \cite{Etnyre04a}.  If $tb(L) = -1$, then according to \cite{MarkTosunPre}, contact $(r)$--surgery on $L$ has vanishing Heegaard Floer invariant for $r < 0$.  Since this invariant must be non-vanishing for strongly symplectically fillable contact structures \cite{Ghiggini06b}, contact $(r)$--surgery on such an $L$ must not be fillable for any $r \in (0, 1)$.

Now, assume that $tb(L) \neq -1$.  Consider a neighborhood $N$ of $L$ with convex boundary, and choose a framing on $L$ so that the dividing curves have slope $0$. After contact $(r)$--surgery on $L$, we have a solid torus $N'$ with Legendrian core $L'$ in the resulting manifold, and using a basis to compute slopes for $\partial N'$ induced from the basis on $\partial N$, we see that $N'$ is a torus with dividing curves of slope $0$ and meridional slope $r$.  As discussed in Section~\ref{transsurg} we know that any slope clockwise of $r$ and counterclockwise of $0$ can be realized as the characteristic foliation on a standard neighborhood of a transverse knot $T$ inside the surgery torus.

If there is an edge between $r$ and $1$ in the Farey tessellation, then the slope $1$ curve is a longitude for the torus which is counter-clockwise of $0$ (the contact framing of $L$), and we can perform an admissible transverse surgery on $T'$ with slope $1$ by attaching a symplectic $2$--handle to $X$ using Theorem~\ref{gaythm}. This will give a weak symplectic filling of the resulting contact manifold, which by the discussion in Section~\ref{transsurg} is the result of contact $(+1)$--surgery on $L$.

If there is not an edge in the Farey tessellation from $r$ to $1$ then choose $s$ between $r$ and $1$ that is closest to $1$ with an edge to $r$. The above argument provides a weak symplectic filling of contact $(s)$--surgery on $L$. As above, we can deform the symplectic structure to be a strong filling, and after doing this a finite number of times, we arrive at a strong symplectic filling of contact $(+1)$--surgery on $L$.  However, given our assumption that $tb(L) \neq -1$, we know that $L$ does not bound a Lagrangian disk in $B^4$, and so as above, contact $(+1)$--surgery on $L$ cannot be strongly symplectically fillable.  Thus, regardless of the value of $tb(L)$, contact $(r)$--surgery on $L$ is not symplectically fillable for any $r \in (0, 1)$.

Finally, the statement about Stein fillings of contact $(+1)$--surgery follows from Corollary~\ref{steincomp}.
\end{proof}

We now turn to larger contact surgeries and prove Proposition~\ref{torusfill} concerning fillability of contact $(r)$--surgeries on Legendrian positive torus knots with maximal Thurston--Bennequin invariant. 
\begin{proof}[Proof of Proposition~\ref{torusfill}]

Smooth $(pq-1)$--surgery on a positive $(p,q)$--torus knot yields a lens space, \cite{Moser71}.  We know that any tight contact structure on a lens space is Stein fillable, \cite{Giroux00,Honda00a}.  In \cite{LiscaStipsicz04}, it was shown that contact $(r)$--surgery on a maximal Thurston--Bennequin invariant $(p,q)$--torus knot $L$ is tight for any $r > 0$.  Since $tb(L)=pq-p-q$ we see that contact $(r-pq+p+q)$--surgery on $L$ corresponds to smooth $r$--surgery on $L$.  Thus contact $(r)$--surgery on $L$ is Stein fillable when $r = p+q-1$.  To show that this also holds for $r > p + q - 1$, we note that by Theorem~\ref{BEtheorem}, for any $r > p+q-1$, we can find a Legendrian link in the result of contact $(p+q-1)$--surgery on $L$ such that doing Legendrian surgery on this link yields the result of contact $(r)$--surgery on $L$.

Moving to the $(2,2n+1)$--torus knot $K_n$, recall that it was shown in \cite{OwensStrle12} that the result of smooth $r$--surgery on $K_n$ does not admit a symplectically fillable contact structure for $r \in [2n-1, 4n)$, so contact $(r)$--surgery on the maximal Thurston--Bennequin invariant representative $L_n$ of $K_n$ with $tb(L_n) = 2n-1$ is not symplectically fillable for $r < 2n+1$.  So to complete the proof of the proposition we need to see that the contact manifold $(M_n,\xi_n)$ coming from contact $(2n+1)$--surgery on $L_n$ is Stein fillable, and the result for $r > 2n+1$ will follow as in the preceding paragraph.  (In \cite{LiscaStipsicz04, OwensStrle12}, it is also shown that the result of smooth $4n$--surgery on $K_n$ admits Stein fillable contact structures, but they are not necessarily the ones coming from contact $(2n+1)$--surgery on $L_n$.)

It is known that $M_1$ has exactly $3$ tight contact structures: two are Stein fillable and one is not symplectically fillable, see Corollary~4.11 and Theorem 4.13 in \cite{GhigginiLiscaStipsicz07}.
 The fillable contact structures have the three-dimensional homotopy invariant $d_3=0$, whereas the non-fillable one has $d_3=-\frac{1}{4}$, ~\cite[Proposition~$4.2$]{GhigginiLiscaStipsicz07}. A simple calculation (see \cite{DingGeigesStipsicz04}) shows that the contact structure $\xi_1$ coming from contact $(3)$--surgery has $d_3(\xi_1)=0$. Thus, $(M_1,\xi_1)$ is clearly Stein fillable.

On the other hand, the manifold $M_{n}$ for $n>1$ admits exactly $4$ tight contact structures, only two of which are known to be Stein fillable \cite[Theorem~$1.1$]{GhigginiLiscaStipsicz07}. In what follows, by using the $d_3$ invariant, we prove that $(M_n,\xi_n)$ is isotopic to one of the Stein fillable contact structures on $M_n$.  To this end, we first briefly recall the construction of all the tight contact structures on $M_n$ from \cite[Section~$4$]{GhigginiLiscaStipsicz07}. One starts with the manifold $M_1$, see Figure~\ref{small}. 
\begin{figure}[h]
\begin{overpic}
{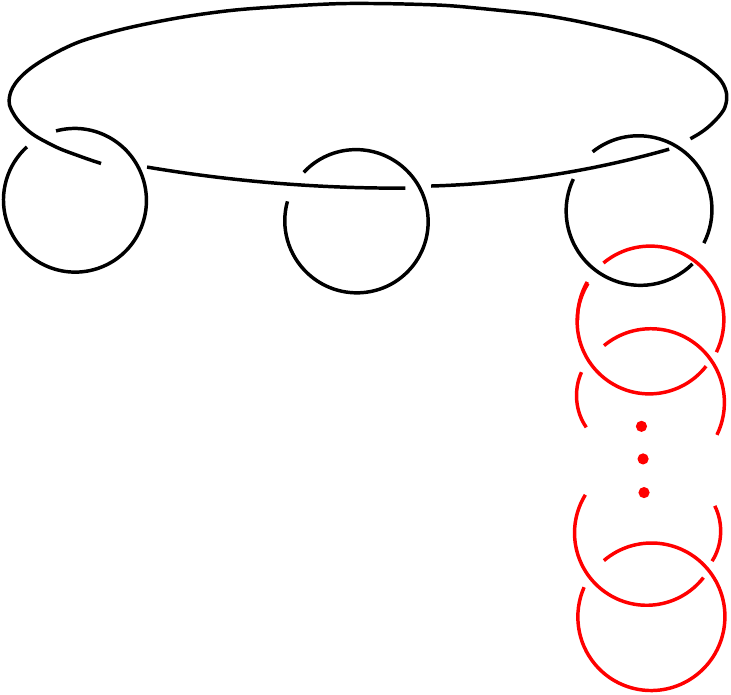}
\put(-14,123){$-2$}
\put(78,108){$-2$}
\put(213, 140){$-3$}
\put(213, 108){$-2$}
\put(213, 82){$-2$}
\put(213, 48){$-2$}
\put(213, 18){$-2$}
\put(100, 205){$-1$}
\put(150, 90){{\color{red} $\mathcal{L}$}}
\end{overpic}
 \caption{The small Seifert-fibered manifold $M_n=M(-1;\frac{1}{2},\frac{1}{2},\frac{n}{2n+1})$, where $n-1$ is the number of terms in the chain of $(-2)$--framed unknots $\mathcal{L}.$ Ignoring $\mathcal L$, we obtain the manifold $M_1$. This diagram also describes the cobordism $W_n$ given by attaching $2$--handles along the chain of the unknots $\mathcal L$. If $M_1$ is given a contact structure, then $W_n$ can be made Stein by attaching Stein $2$--handles along the unknots realized as Legendrian knots with maximal Thurston--Bennequin number.}
  \label{small}
\end{figure}
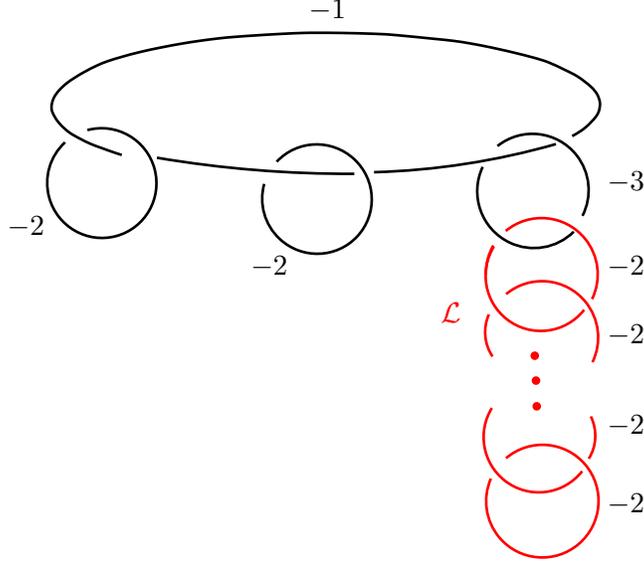 
As mentioned above, on $M_1$ there are exactly three tight contact structures, up to isotopy. Two of these, $\xi'_1, \xi'_2$, are Stein fillable, whereas the third $\Xi$ is not fillable (see~\cite[Figures~$6$, $8$]{GhigginiLiscaStipsicz07} for their explicit contact surgery diagrams). One then performs Legendrian surgery on a chain of $n-1$ Legendrian unknots with maximal Thurston--Bennequin number in each of $\xi'_1$, $\xi'_2$, $\Xi$ and $\Xi'$ (here $\Xi'$ denotes the conjugate of $\Xi$, see \cite[Remark~$4.10(2)$]{GhigginiLiscaStipsicz07}) to obtain contact structures $\eta_1$, $\eta_2$, $\theta_1$, and $\theta_2$, respectively. It was shown in \cite[Theorem~$1.1$]{GhigginiLiscaStipsicz07} that these contact structure constitute the complete list of tight contact structures on $M_n$ for any $n>1$. Note that $\eta_1$ and $\eta_2$ are Stein fillable, as Legendrian surgery preserves Stein fillability. It is unclear, however, if either of $\theta_1$ or $\theta_2$ is fillable. 
 
Below we will compute that
\begin{align*}
d_3(\eta_1)&=d_3(\eta_2)=\frac{n-1}{4},\\
 d_3(\theta_1)&=d_3(\theta_2)=\frac{n-2}{4}, \text{ and}\\
 d_3(\xi_n)&=\frac{n-1}{4}.
\end{align*}
Since $\xi_n$ is a tight contact structure on $M_n$, by \cite{LiscaStipsicz04}, and $\eta_1$, $\eta_2$, $\theta_1$, and $\theta_2$ are the only tight structures on $M_n$, up to isotopy, $\xi_n$ must be isotopic to one of $\eta_1$ and $\eta_2$. In particular, $(M_n,\xi)$ is Stein fillable. 

We start by determining $d_3(\xi_n)$. Recall that $\xi_n$ is the contact structure on $M_n$ obtained by contact $(2n+1)$--surgery on the Legendrian $(2,2n+1)$--torus knot $L_n$ with $tb(L_n)=2n-1$. As discussed in Section~\ref{surgery}, contact $(2n+1)$--surgery on $K$ is effected by taking a Legendrian push-off $L_n'$ of $L_n$, stabilizing it once negatively, and taking $2n-1$ Legendrian push-offs of $L_n'$ (for a total of $2n$ copies of $L_n'$). Now perform contact $(+1)$--surgery on $L_n$ and contact $(-1)$--surgery on all the copies of $L_n'$. Notice that each copy of $L_n'$ has rotation number $-1$. If $X_n$ is the cobordism from $S^3$ to $M_n$ obtained by attaching 2--handles to the Legendrian knots with the above prescribed framings, then \cite{DingGeigesStipsicz04} gives the formula 
\[
d_3(\xi_n)=\frac 14 (c^2-3\sigma(X_n)-2\chi(X_n))+1
\]
for the $d_3$ invariant, where $c\in H^2(X,\partial X)$ evaluates on the co-core of a 2--handle to be the rotation number of the corresponding Legendrian knot, $\sigma(X_n)$ is the signature of $X_n$, and $\chi(X_n)$ is the Euler characteristic. One may compute $c=-n$, $\sigma(X_n)=1-2n$, and $\chi(X_n)=2n+1$. So $d_3(x_n)=\frac{n-1}4$ as claimed.

Turning now to the other contact structures, we first recall that by ~\cite[Proposition~$4.2$]{GhigginiLiscaStipsicz07} we have $d_3(\xi'_1)=d_3(\xi'_2)=0$ and $d_3(\Xi)=d_3(\Xi')=-\frac{1}{4}$. Let $W_n$ denote the Stein cobordism between $(M_1,\xi'_1)$ and $(M_n,\eta_1)$. Since $M_1$ is a rational homology sphere, it is easy to see the $d_3$ invariant changes by the value $\frac 14 (c^2-3\sigma(W_n)-2\chi(W_n))$, and so we compute that
\[
d_3(\eta_1)=d_3(\xi'_1) + \frac 14 (0-3(-n+1) - 2(n-1))=\frac{n-1}4.
\]
Similar calculations for the other contact structures show that $d_3(\eta_2)=\frac{n-1}{4}$ and $d_3(\theta_1)=d_3(\theta_2)=\frac{n-2}{4}$.
\end{proof}

We now turn to the obstructions for symplectic fillability of positive contact surgery.

\begin{proof}[Proof of Theorem~\ref{blowupdisk}]
We first assume that $r = n$ is an integer.  Let $L \subset (S^3, \xi_{std})$ be a Legendrian knot, and $n > 0$ an integer such that contact $(r)$--surgery on $L$ is symplectically filled by $(X, \omega)$.  Consider a neighborhood $N$ of $L$ with convex boundary, and choose a framing on $L$ such that the dividing curves on $\partial N$ have slope $0$.  Then, as in the proof of Theorem~\ref{thm1}, we have a solid torus $N'$ in $\partial X$ with convex boundary, and using a basis for $\partial N'$ induced from the basis on $\partial N$ to compute slopes, we see that $N'$ is a solid torus that has a convex boundary with dividing curves of slope $0$ and meridional slope $n$.  There is a transverse knot $T$ isotopic to the core of $N'$, and any slope clockwise of $n$ and counterclockwise of $0$ can be realized as the characteristic foliation on a standard neighborhood of $T$.  In particular, the $\infty$ slope is an admissible slope, and we can perform admissible transverse surgery on $T$ with slope $\infty$ to arrive back at $S^3$.

We can assume that $X$ has a strongly convex boundary: if the smooth surgery coefficient $tb(L) + n$ is non-zero, then we can purturb the filling to have strongly convex boundary.  If not, we replace $n$ by $n+1$: we first find a Legendrian knot $L'$ in $N'$ such that Legendrian surgery on $L'$ gives the contact manifold resulting from contact $(n+1)$--surgery on $L$, by Theorem~\ref{BEtheorem}.  We can attach a Weinstein $2$--handle to $X$ along $L'$ to construct a weak filling of contact $(n+1)$--surgery on $L$, by \cite{EtnyreHonda02b}, and since $tb(L) + n + 1 \neq 0$, this filling can be purturbed to have strongly convex boundary.

Now since we assume that the filling has convex boundary, by Theorem~\ref{gaythm}, we can effect the admissible transverse surgery by attaching a symplectic $2$--handle to $T \subset \partial X$, and construct a weak symplectic filling $(X', \omega')$ of $S^3$.  By \cite{Gromov85, McDuff90}, we know that $(X', \omega')$ must be $B^4$ with its standard symplectic structure blown-up some number of times.  Since by \cite[Theorem~5]{Wendl2013}, the co-core of the 2--handle is a symplectic disk with boundary the transverse push-off of $L$, we see that the transverse push-off of $L$ bounds a disk in $B^4 \# m\overline{\C P}^2$, for some $m$.

If $r > 0$ is a rational number, let $n$ be the smallest integer such that $n \geq r$.  Then, we reduce the rational case to the integer case by noticing that the argument above shows that if contact $(r)$--surgery on $L$ is symplectically fillable, then the result of contact $(n)$--surgery on $L$ is likewise symplectically fillable.
\end{proof}

\begin{proof}[Proof of Theorem~\ref{taubound}]
Continuing with the discussion from the previous proof, consider a smooth knot $K$ that bounds a disk $D$ in $X = B^4 \# m\overline{\C P}^2$.  Choose a basis $\{e_i\}$ for $H_2(X, \partial X)\cong H_2(X)$ such that $e_i \cdot e_j = -\delta_{ij}$.  If $[D] = \sum d_ie_i$, then we claim that removing a neighborhood of $D$ results in a 4--manifold whose boundary is the result of smooth surgery on $K$ with surgery coefficient $\sum d_i^2$.  To see this, note if $\Sigma$ is a Seifert surface for $\partial D$ then $D\cup -\Sigma$ represents the homology class $[D]$ in $H_2(X)\cong H_2(X,\partial X)$, thus $[D\cup -\Sigma]^2=-\sum d_i^2$. So if we take the unique framing of the normal bundle of $D$ in $X$, then it induces a framing on $\partial D$ in $S^3$ for which the push-off of $\partial D$ by the framing intersects $-\Sigma$ exactly $-\sum d_i^2$ times. Thus, the linking of the push-off of $\partial \Sigma$ with $\partial D$ is $\sum d_i^2$. In particular, removing a neighborhood of $D$ from $X$ will affect the boundary by removing a neighborhood of $\partial D$ from $S^3$ and then gluing in a solid torus with meridian going to the longitude with framing $\sum d_i^2$, as claimed.

This observation will allow us to define the function $f$.  According to \cite{OzsvathSzabo03}, if $X$ is a 4--manifold with boundary $S^3$ and $b_2^+(X) = b_1(X) = 0$, and $K \subset S^3$ is a knot that bounds a smooth surface $\Sigma$ in $X$, then $$2\tau(K) + \bigl|[\Sigma]\bigr| + [\Sigma]\cdot[\Sigma] \leq 2g(\Sigma),$$ where if $\{e_i\}$ is an orthonormal basis for $H_2(X, \partial X)$, and $[\Sigma] = \sum d_i e_i$, then $\bigl|[\Sigma]\bigr| = \sum |d_i|$ and $[\Sigma]\cdot[\Sigma] = \sum d_i^2e_i^2$.

In our case, $X = B^4 \# m\overline{\C P}^2$ satisfies the requirements, and $D$ is a surface of genus $g(D) = 0$.  Thus, since $e_i^2 = -1$, we require that $$2\tau(K) \leq \sum\left(d_i^2 - |d_i|\right).$$  We then define $f(\tau)$ to be the minimum of $\sum d_i^2$ over all $m > 0$ and tuples $(d_1, \ldots, d_m) \in \mathbb Z^m$ satisfying the required inequality.  To get the form of $f$ from the introduction, note that the minimum will occur when $d_i \geq 0$ for all $i$ and thus we can leave off the absolute value signs.

Now, contact $(r)$--surgery on a Legendrian knot $L$ results in the same manifold as smooth surgery on $L$ with framing $tb(L) + r$.  Let $n$ be the smallest integer satisfying $n \geq r$; as in the proof of Theorem~\ref{blowupdisk}, if contact $(r)$--surgery is fillable, then so is contact $(n)$--surgery.  Then, by Theorem~\ref{blowupdisk}, we can find a smooth disk $D$ in $B^4 \# m \overline{\C P}^2$ for some $m \geq 0$ giving $\sum d_i^2 = tb(L) + n$.  According to the definition of $f(\tau)$, this cannot happen unless $tb(L) + n \geq f(\tau(L))$.
\end{proof}
\def\cprime{$'$}

\end{document}